\documentclass{amsart}

\usepackage{amsmath,amsfonts,amsthm,amssymb,amscd}
\usepackage{hyperref}
\usepackage{xcolor}
\usepackage[alphabetic]{amsrefs}

\usepackage{tikz}
\usepackage{graphicx}

\theoremstyle{plain}
\newtheorem{theorem}{Theorem}[section]
\newtheorem{proposition}[theorem]{Proposition}
\newtheorem{step}{Step}
\newtheorem{lemma}[theorem]{Lemma}
\newtheorem{question}[theorem]{Question}
\newtheorem{corollary}[theorem]{Corollary}

\theoremstyle{definition}

\newtheorem{remark}[theorem]{Remark}

\numberwithin{equation}{section}

\newcommand{\Z}{\mathbb{Z}}   
\newcommand{\N}{\mathbb{N}}   
\newcommand{\abs}[1]{\left \lvert#1 \right \rvert}

\newcommand{\con}{l} 
\newcommand{\growth}{\mathsf{Growth}}

\newcommand{\nclose}[1]{\ensuremath{\langle\!\langle#1\rangle\!\rangle}}

\DeclareMathOperator{\dist}{\mathsf{dist}}
\DeclareMathOperator{\diam}{\mathsf{diam}}

\title[Coarse geometry of fire retaining and group splittings]{Coarse geometry of the fire retaining property and group splittings}

\author[E. Mart\'inez-Pedroza]{Eduardo Mart\'inez-Pedroza}

\address{Department of Mathematics and Statistics. 
Memorial University of Newfoundland. St. John's, NL. Canada}

\email{eduardo.martinez@mun.ca}

\author[T. Prytu{\l}a]{Tomasz Prytu{\l}a}

\address{
Alexandra Instituttet
Rued Langgaards Vej 7, 5D, 2300 København S, Denmark}

\email{prytula@mpim-bonn.mpg.de}

\subjclass[2000]{Primary 05C57, 20F65; Secondary 05C10, 20F69}

\keywords{games on graphs, quasi-isometry, splitting, coarse geometry, growth}
 
\date{\today}

\begin{document}

	\begin{abstract}
		Given a non-decreasing function $f\colon \N\to \N$ we define a single player game on (infinite)  connected graphs that we call \emph{fire retaining}.  If a graph $G$  admits a winning strategy for any initial configuration (initial fire) then we say that $G$ has the \emph{$f$-retaining property}; in this case if $f$ is a polynomial of degree $d$, we say that $G$ has the polynomial retaining property of degree $d$. 

		We prove that having the polynomial retaining property of degree $d$ is a quasi-isometry invariant in the class of uniformly locally finite connected graphs. Henceforth, the retaining property defines a quasi-isometric invariant of finitely generated groups. We prove that if a finitely generated group $G$ splits over a quasi-isometrically embedded subgroup of polynomial growth of degree $d$, then $G$ has polynomial retaining property of degree $d-1$. Some connections to other work on quasi-isometry invariants of finitely generated groups are discussed and some questions are raised.
	\end{abstract}

	\maketitle

	\section{Introduction}

	The firefighter problem on graphs was introduced by Hartnell in 1995 and it has been studied by graph theorists ever since, see for example~\cite{FM09} and references therein. Recently, in~\cite{DMT17}, this problem was studied in the context of coarse geometry, leading to definition of quasi-isometry invariants of finitely generated groups that were called \emph{fire containment} properties. The current article follows the same vein: we study a variation that we call \emph{fire retaining} properties, we define new quasi-isometry invariants of finitely generated groups, and we exhibit a relation with the existence of group splittings.

	Let $\{f_n\}_{n\in\N}$ be a sequence of non-negative integers. Suppose that a fire breaks out at a finite set of vertices $X_0$ of a connected graph $G$. At each subsequent time unit $n \in \N$ (called a turn), the player (called the firefighter) chooses a set $W_n$ of at most $f_n$  distinct vertices to become protected; then the fire spreads to all  vertices which are adjacent to vertices which are on fire and are not yet protected. Once a vertex is on fire or is protected, it stays in such state for all subsequent turns.  Denote by $U$ the set of vertices which, after the game has been played, never caught fire. 

	\begin{itemize}
		\item If $U$ contains all but finitely many vertices of $G$, we say that the sequence $\{W_n\}$ is a \emph{containment $\{f_n\}$-strategy} for the initial fire $X_0$.

		\item If the growth rate of $U$ with respect to the edge-path distance in $G$ is equivalent to the growth rate of $G$,  we say that the sequence $\{W_n\}$ is a \emph{retaining $\{f_n\}$-strategy} for the initial fire $X_0$.
	\end{itemize}

	If every finite subset of vertices $X_0$ of $G$ admits a containment $\{f_n\}$-strategy we say the graph $G$ satisfies the \emph{$\{f_n\}$-containment property}. The graph $G$ satisfies \emph{polynomial containment of degree $d$} if there is a constant $K>0$ such that $G$ has the $\{f_n\}$-containment property for $f_n=Kn^d$. The \emph{$\{f_n\}$-retaining property} and \emph{polynomial retaining property of degree $d$} for a connected graph are defined analogously. 

	In Figure~\ref{fig:map} we present a relation between various types of retaining and containment properties, along with some examples of graphs (Cayley graphs of finitely generated groups) satisfying these properties.

	Observe that if a connected graph satisfies the $\{f_n\}$-containment property then it satisfies the $\{f_n\}$-retaining property. In~\cite{DMT17}, it was proved that satisfying the polynomial containment property of degree $d$ is a quasi-isometry invariant in the class of uniformly locally finite graphs (uniformly locally finite means there is a constant that bounds from above the degree of any vertex in the graph). The first theorem in this article says that the analogous result holds for the polynomial retaining property.

	\begin{theorem}(Corollary~\ref{thm:qi})\label{introthm:1}
		Let $G$ and $H$ be uniformly locally finite connected graphs. 
		Suppose that $G$ is quasi-isometric to $H$. If $G$ has polynomial retaining property of degree $d$  then $H$ has polynomial retaining property of degree $d$.  
	\end{theorem}

	A more general version of Theorem~\ref{introthm:1} is the main content of Section~\ref{sec:qi}.
	Since any two Cayley graphs of a finitely generated group $G$ with respect to finite generating sets are quasi-isometric, satisfying polynomial retaining or containment of degree $d$ is a well-defined invariant of  finitely generated groups. 
	
	We search for algebraic interpretations of these properties in the class of finitely generated groups. For example, there are known relations between containment and growth for finitely generated groups. In~\cite{DMT17}, it is proved that if a group has polynomial growth of degree $d$ then it satisfies polynomial containment of degree $d-2$. Recently, 	    joint work of Amir, Baldasso and Kozma~\cite{ABK20} shows  that any Cayley graph with  growth bounded from below by a polynomial of degree $d$ does not satisfy the $\{f_n\}$-containment property if $\{f_n\}=o(n^{d-2})$.
	It is know that finitely generated groups of intermediate growth do not satisfy polynomial containment~\cite{ABK20}. See also~\cite{Le17, emp16, ABK20} for other related results.
	
	The second result of this article exhibits a relation between retaining properties and splittings of groups.  We say that a group \emph{$G$ splits over a subgroup $C$} if either $G=A\ast_C B$ and $C$ is a proper subgroup of $A$ and $B$, or $G$ is an HNN-extension $A\ast_C$ (with no assumptions on $C$). 

	\begin{theorem}(Theorem~\ref{thm:splitpoly})\label{introthm:2}
		Let $G$ be a finitely generated group that splits over a finitely generated subgroup $C$. If $C$ is quasi-isometrically embedded in $G$ and $C$ has polynomial growth of degree $d$, then $G$ has polynomial retaining property of degree $d-1$.
	\end{theorem}

	\begin{proof}[Outline of the proof of Theorem~\ref{introthm:2}.] 
		The splitting ensures that the Cayley graph of $G$ can be disconnected into two unbounded components by removing an appropriate neighborhood $L$ of $C$. Using $L$ we can build a ``wall'' in $G$, namely by protecting all vertices of $L$.

		To do this, one has to ensure that in the process of protecting $L$, one is always ahead of the spreading fire. 
		Since $C$ is quasi-isometrically embedded, its growth inside $G$ is polynomial of degree $d$. Since $L$ and $C$ are quasi-isometric, the same holds for $L$. Therefore at time $n$, at most (roughly) $n^d$ vertices of $L$ could potentially catch fire. One verifies that one can protect $n^d$ vertices by protecting at time $k$ for $1\leq k \leq n$ an amount of vertices that grows polynomially of degree $d-1$ with $k$. 
		Because $L$ disconnects the Cayley graph of $G$, we get that an  entire unbounded component will never catch fire. Then the homogeneity of the Cayley graph implies that this component has the growth rate as large as the group.  
	\end{proof}

	A finitely generated group $G$ has the \emph{constant retaining property} if it has polynomial retaining property of degree zero. The following is  a particular instance of Theorem~\ref{introthm:2}.

	\begin{corollary}
		Suppose $G = A \ast_C B$ where $C$ is a proper subgroup of $A$ and $B$. If $C$ is virtually cyclic and quasi-isometrically embedded in $G$, then $G$ has the constant retaining property.
	\end{corollary}

	The quasi-isometry invariance of splittings of finitely presented groups over two-ended (i.e.,\ virtually cyclic) groups was settled by deep results of Papasoglu~\cite{Pa05}. In particular he shows that if $G$ is a one-ended, finitely presented group that is not commensurable to a surface group, then $G$ splits over a two-ended group if and only if the Cayley graph of $G$ with respect to a finite generating set is separated by a quasi-line. We refer the reader to~\cite{Pa05} for the definitions of quasi-line and separation. 

	\begin{question} 
		Let $G$ be a one-ended, finitely presented group, not commensurable to a surface group. Suppose that $G$ has the constant retaining property.  Is the Cayley graph of $G$ with respect to a finite generating set separated by a quasi-line?
	\end{question}

	While constant containment implies constant retaining, the converse does not hold, as exhibited for example by the free group of rank two~\cite{DMT17}. More interesting examples are the free abelian groups of rank at least three; the proof that these groups do not satisfy the constant containment property is due to Develin and Hartke~\cite{DH07}.  We suspect that these groups do not satisfy the constant retaining property either.

	\begin{question}\label{question:zcube}
		Does the group $\Z^3$ have the constant retaining property?
	\end{question}

	A connected graph has the \emph{finite-step polynomial retaining property of degree $d$} if there is a polynomial $f$ of degree $d$ such that  any initial fire $X_0$ admits a retaining $f$-strategy $\{W_n\}$ such that $W_n$ is empty for all but finitely many $n$.  A corollary of the proof of Theorem~\ref{introthm:1} is that the finite-step polynomial retaining property of degree $d$ is a quasi-isometry invariant in the class of uniformly locally finite graphs, see Corollary~\ref{cor:qi-fsrp}.  In regard to Question~\ref{question:zcube} above, the group $\Z^3$ does have the finite-step retaining property of degree one, see Remark~\ref{rem:differentd}.

	The finite-step polynomial retaining property of degree zero is abbreviated as the \emph{finite-step retaining property}. For finitely generated groups, this property essentially captures splittings over finite subgroups.

	\begin{theorem}\label{thm:finite-step-retaining}
		Let $G$ be a finitely generated group.
		\begin{enumerate}
			\item \label{item:1thm-finite-step} If $G$ splits over a finite group, then $G$ has the finite-step retaining property. 
			\item \label{item:2thm-finite-step} If $G$ has the finite-step retaining property, then $G$ has the constant containment property or $G$ splits over a finite group. 
		\end{enumerate}
	\end{theorem}

	The two statements of Theorem~\ref{thm:finite-step-retaining} correspond to Proposition~\ref{cor:split-finite} and Corollary~\ref{cor:Tomasz}, respectively, in the main body of the article. 


	\begin{corollary}
		Let $G$ be a non-amenable finitely generated group. Then $G$ is one-ended if and only if $G$ does not have the finite-step retaining property. 
	\end{corollary}

	\begin{proof}
		Since $G$ is non-amenable, it does not satisfy the constant containment property~\cite[Corollary~8]{emp16}, and $G$ has either one or infinitely many ends~\cite[Part~I, Theorem~8.32(1,2,3)]{BH99}. 

		Suppose that $G$ has infinitely many ends. By Stallings' theorem~\cite[Part~I, Theorem~8.32(5)]{BH99}, $G$ splits over a finite group and hence Theorem~\ref{thm:finite-step-retaining}\eqref{item:1thm-finite-step} implies that $G$ has the finite-step retaining property.
		Conversely, by Theorem~\ref{thm:finite-step-retaining}\eqref{item:2thm-finite-step}, if $G$ has the finite-step retaining property then $G$ has infinitely many ends. 
	\end{proof}

	Regarding the statement of Theorem~\ref{introthm:2}, we believe that  the hypothesis that the subgroup $C$ is quasi-isometrically embedded could be weakened. For example, we expect a positive answer to the following.

	\begin{question}
		Let $G$ be a finitely generated group isomorphic to an amalgamated product $A\ast_C B$. Suppose that $A$ is hyperbolic relative to $C$, and that $C$ has polynomial growth of degree $d>0$.  Does $G$ have polynomial retaining property of degree $d-1$?
	\end{question}
	

	\subsubsection*{Organization} 
	
	Section~\ref{sec:preliminaries} contains some preliminary material on the notion of growth rate of graphs and quasi-isometry of metric spaces. Detailed definitions of the retaining property, containment property and finite-step retaining properties, as well as some preliminary results, are the content of Section~\ref{sec:GameDefinitions}. The proofs
	of Theorems~\ref{introthm:1} and~\ref{introthm:2} are given in Sections~\ref{sec:qi} and~\ref{seq:splitting} respectively. 

	\subsubsection*{Acknowledgements}
	Both authors thank the referees for corrections and suggestions on the manuscript.
	E.\ M.\ P.\ acknowledges partial funding by  the Natural Sciences and Engineering Research Council of Canada (NSERC). T.\ P.\ was supported by the EPSRC First Grant EP/N033787/1.

	\section{Preliminaries}\label{sec:preliminaries}
	 
	\subsection{Growth rate}\label{def:growthrate} 
	
	Given non-decreasing functions $f \colon \N\to \N$ and  $g \colon \N \to \N$, the relation $f \preceq g$  is defined as the existence of an integer $C > 0$ such that \[f(n) \leq C g(Cn + C) + C\] for every $n$. 
	The functions $f$ and $g$ have \emph{equivalent growth rate}, denoted by $f\sim g$, if $f \preceq g$ and $g \preceq f$.  For a locally finite metric space $(X, \dist)$, the \emph{growth function} $\beta_{X,A} \colon \N \to \N$ with respect to a non-empty finite subset $A\subset X$ is defined as \[\beta_{X,A} (n) = \abs{ B_X(A,n)},\] where \[B_X(A,n)=\{x\in X \colon \dist(A,x)\leq n\}.\] Observe that for any two finite subsets $A,B$ of $X$ the growth functions $\beta_{X,A}$ and $\beta_{X,B}$  have equivalent growth rate. The \emph{growth rate of $X$}, denoted by $\growth (X)$, is the equivalence class of $\beta_{X,A}$ with respect to the equivalence relation $\sim$. 
	For locally finite metric spaces $X$ and $Y$,  define $\growth (X) \preceq \growth (Y)$ if $\beta_{X,A} \preceq \beta_{Y,B}$ for some (and hence for any) choices of finite subsets $A\subseteq X$ and $B \subseteq Y$.

	\begin{remark}\label{rem:growth}
		The following statements are easy to verify. 
		\begin{enumerate}
			\item \label{it:growth1} If $\growth (X) \preceq \growth (Y)$  and $\growth (Y) \preceq \growth (X)$  then  $\growth (X) = \growth (Y)$. 

			\item \label{it:growth2} Let $U$ be a subset of a locally finite metric space $X$. Consider $U$ as a metric space with the metric induced from $X$. Then $\growth (U) = \growth (X)$ if and only if $\growth(X) \preceq \growth(U)$.

			\item \label{it:growth3}  A locally finite metric space  $X$ is \emph{uniformly locally finite} if there is a function $g\colon \N \to \N$ such that for any $x\in X$ the ball $B_X(x,n)$ has cardinality at most $g(n)$. If $X$ and $Y$ are quasi-isometric   uniformly locally finite metric spaces then $\growth (X) = \growth (Y)$ (the defintion of \emph{quasi-isometry} is recalled in Subsection~\ref{subsec:quasiiso} below). 
		\end{enumerate}
	\end{remark}

	\subsection{Graphs as metric spaces} 

	Let $G$ be a graph. We say that $G$ is \emph{uniformly locally finite} if there is a constant $M$ such that every vertex of $G$ is adjacent to at most $M$ vertices. A \emph{path of length $n$} is a sequence of vertices $v_0,v_1, \ldots , v_n$ such that $v_i, v_{i+1}$ are connected by an edge for each $i<n$. The graph $G$ is \emph{connected} if there is a path between any two vertices of $G$.  Assume that $G$ is connected. The set of vertices of $G$ is denoted by $V(G)$. The notion of path defines a metric on the set of vertices of $G$ by declaring $\dist_G(x, y)$ to be the length of the shortest path from $x$ to $y$; we call this metric the \emph{edge-path  metric}.  

	Let $X$ and $Y$ be subsets of $V(G)$. The \emph{ball of radius $r$ centered at $X$}, denoted by $B_G(X,r)$, is defined as the collection of vertices at distance less than or equal to $r$ from at least one vertex in $X$. The distance $\dist_G(X,Y)$ is defined as the minimum of distances $\dist_G(x,y)$ where $x \in X$ and $y \in Y$. The \emph{diameter of $X$} denoted by $\diam X$ is defined as $\sup\{ \dist (x_1, x_2) \colon x_1,x_2 \in X \}$. 

	\subsubsection*{Growth in graphs}

	If a graph $G$ is uniformly locally finite then the set of vertices of $G$ with the edge-path metric $\dist_G$ is a   uniformly locally finite metric space. Define $\growth (G)$ as the growth rate of 
	$(V(G), \dist_G)$. For any subset $U \subset V(G)$, we denote by $\growth (U)$ the growth rate of $U$ with the metric being the restriction of $\dist_G$ to $U$.  

	\subsection{Quasi-isometry}\label{subsec:quasiiso}

	Let $(X,\dist_X)$ and $(Y,\dist_Y)$ be metric spaces and let $C>0$ be a constant. A map $\phi \colon X\to Y$ is a \emph{$C$-quasi-isometric embedding}  if  for all $x_1,x_2\in X$ we have
	\[ \frac1C \dist_X (x_1,x_2) - C \leq \dist_Y (\phi(x_1), \phi(x_2)) \leq C \dist_X (x_1,x_2)+C.\]

	A $C$-quasi-isometric embedding $\phi\colon X \to Y$ is  a \emph{$C$-quasi-isometry} if every point of $Y$ lies in the $C$-neighborhood of the image of $\phi$.  We say that the metric spaces $X$ and $Y$ are \emph{quasi-isometric} if there is a $C$-quasi-isometry from $X$ to $Y$ for some constant $C$. 

	Note that the identity function on a metric space is a $1$-quasi-isometry and that the composition of a $C$-quasi-isometry  with a $C'$-quasi-isometry is a $C''$-quasi-isometry for some $C''$ that depends only on $C$ and $C'$. Moreover, if $\phi \colon X\to Y$ is a $C$-quasi-isometry then there is a $C'$-quasi-isometry $\psi\colon Y \to X$ such that $\dist_X( x, \psi\circ \phi (x)) \leq C'$ for all $x\in X$, see~\cite[page~138]{BH99}. 

	Let $G$ and $H$ be connected graphs. A \emph{$C$-quasi-isometry $\phi\colon G\to H$} is a $C$-quasi-isometry from the vertex set of $G$ with its edge-path metric into the vertex set of $H$ with its edge-path metric.  We say that the graphs $G$ and $H$ are \emph{quasi-isometric} if their vertex sets with their corresponding  edge-path metrics are quasi-isometric metric spaces.  A connected subgraph $H$ of the connected graph $G$ is \emph{quasi-isometrically embedded} if the corresponding inclusion map on the vertex sets is a quasi-isometric embedding with respect to the edge-path metrics. 

\subsection{Finitely generated Groups and Cayley Graphs}
Let $G$ be a group with a finite generating set $S$. The \emph{Cayley graph $\Gamma(G,S)$} is the graph with vertex set $G$  and edge set $\{\{g,gs\}\colon g\in G \text{ and }  s\in S\}$. The natural left action of $G$  on $\Gamma(G,S)$ has finitely many orbits of vertices and edges. Two finitely generated groups are quasi-isometric if their Cayley graphs with respect to some (and hence for any) finite generating sets are quasi-isometric. For a fixed finite generating set of a group, the edge-path metric on the corresponding Cayley graph and the induced \emph{word-metric} on the group coincide. A finitely generated subgroup $C$ of  a finitely generated group $G$ is \emph{quasi-isometrically embedded} if for some (and hence any) finite generating set $S$ of $G$ containing a finite generating set $T$ of $C$ the Cayley graph of $C$ with respect to $T$ is quasi-isometrically embedded in the Cayley graph of $G$ with respect to $S$. For a detailed discussion of these matters we refer the reader to~\cite{BH99}.

	\section{The fire retaining property for graphs}\label{sec:GameDefinitions}

	In this section we define the retaining property, containment property and finite-step retaining properties. In Figure~\ref{fig:map} we show how these properties relate to each other and we give some examples of graphs satisfying these properties.\medskip 

	Let $G$ be a connected graph. Let $r>0$ be a positive integer that we shall call the \emph{fire reach}. Let $\{f_n\}_{n\geq 1}$ be a sequence of positive integers that we shall call the \emph{strategy bound}. The player of the game shall be called the \emph{firefighter}. 
	Let $X_0$ be a finite subset of vertices of $G$ that we shall call the \emph{initial fire}. 

	\subsection{Strategies} 

	A \emph{$\{f_n\}_{n\geq 1}$-strategy} is a sequence $\{W_n\}_{n\geq 1}$ of subsets of vertices of $G$ such that  for every $n\geq 1$, the set $W_n$ has cardinality at most $f_n$.  The set $W_n$ is called the \emph{set of vertices to protect at time $n$}.  If the sequence $\{f_n\}_{n\geq 1}$ is constant, i.e.,\ if $f_n=f$, then an $\{f_n\}_{n\geq 1}$-strategy is called an \emph{$f$-strategy}.

	\subsection{Vertices on fire at time $n$}\label{def:Xn}

	Now we define \emph{the set $X_n$ of vertices on fire at time $n$ with respect to the $\{f_n\}$-strategy $\{W_n\}_{n\geq 1}$, and the initial fire $X_0$ of reach $r$.} In words, the set $X_n$ consists of all the vertices of $G$ that can be reached from a vertex of $X_{n-1}$ by a path of length at most $r$, which avoids all vertices that have been protected up to time $n$. Since a vertex that is on fire at some time of the game  remains on fire for the rest of the game, the set of vertices that are protected at time $n$ is \[(W_1\cup \cdots \cup W_n) \setminus X_{n-1}.\] Formally, for each integer $n>0$, the subset $X_n$ consists of vertices which are connected to a vertex of $X_{n-1}$ by a path of length at most $r$ containing no vertices in $(W_1\cup \cdots \cup W_n) \setminus X_{n-1}$.  

	The set $\bigcup_{n\geq 0} X_n$ shall be called \emph{the set of vertices on fire at the end of the game with respect to the $\{f_n\}$-strategy $\{W_n\}_{n\geq 1}$, and the initial fire $X_0$ of reach $r$.}

	\subsection{Equivalent strategies}

	The $\{f_n\}_{n\geq 1}$-strategies $\{W_n\}_{n\geq 1}$ and $\{V_n\}_{n\geq 1}$ are \emph{equivalent for the initial fire $X_0$ of reach $r$} if the corresponding sets of vertices on fire at time $n$ for both strategies are equal for every $n\geq 0$.

	\begin{remark} 
		Let $\{W_n\}_{n\geq 1}$ be a strategy and let $X_0$ be an initial fire of reach $r$. Let $X_n$ denote the set of vertices on fire at time $n$ for the given data. Observe that the definitions above do not imply that $X_{n}\cap W_{n+1}=\emptyset$. In words, at time $n+1$, the firefighter might be unable to protect a vertex $v$ in $W_{n+1}$ because $v$ caught fire at an earlier stage of the game. This can be avoided by passing to an equivalent strategy, as the following lemma states.
	\end{remark}

	\begin{lemma}\label{lem:efficient}
		Let $X_n$ be the set of vertices on fire at time $n$ with respect to the $\{f_n\}$-strategy $\{W_n\}_{n\geq 1}$ and initial fire $X_0$ of reach $r$. Then there is an $\{f_n\}$-strategy $\{W'_n\}_{n\geq 1}$ equivalent to $\{W_n\}_{n\geq 1}$ such that $X_n \cap W'_{n+1} =\emptyset$ for every $n$. Specifically, $W'_{n+1}=W_{n+1}\setminus X_n$ for all $n>0$.
	\end{lemma}

	The proof is straightforward and is left to the reader.

	\subsection{Retaining strategies}

	Given an $\{f_n\}$-strategy $\{W_n\}_{n\geq 1}$ and the initial fire $X_0$ of reach $r$, define $U$ to be the complement of $\bigcup_{n=0}^\infty X_n$ in the vertex set of $G$. Thus $U$ is the set of vertices of $G$ that at the end of the game are not on fire. 

	The strategy $\{W_n\}_{n\geq 1}$ is called a \emph{retaining $\{f_n\}$-strategy for the initial fire $X_0$ of reach $r$} if $\growth(U) = \growth(G)$ where the metric on $U$ is the restriction of the edge-path metric on the vertex set of $G$.  

	If we wish to emphasize the reach of the fire, we will write that $\{W_n\}_{n\geq 1}$ is a \emph{retaining $(\{f_n\},r)$-strategy for} $X_0$.

	\subsection{Retaining property} 

	The graph $G$ has the \emph{$(\{f_n\}, r)$-retaining  property} if for every finite subset $X_0$ of vertices of $G$ there is a retaining $\{f_n\}$-strategy for $X_0$ as the initial fire of reach $r$. 

	We will use the following abbreviations:
	\begin{enumerate}
		\item  $G$ has the \emph{$\{f_n\}$-retaining property} means that $G$ has the $(\{f_n\}, r)$-retaining  property for $r=1$.
		\item  $G$ has \emph{polynomial retaining property of degree $d$} means that there is a constant $K>0$ such that $G$ has the $\{Kn^d\}$-retaining property.
		
		\item  $G$ has the \emph{$f$-retaining property} means that $G$ has the $\{f_n\}$-retaining  property for the constant sequence $f_n=f$. 
		
		\item  $G$ has \emph{constant retaining property} means that $G$ has the $f$-retaining property for some integer $f$, or equivalently, $G$ has  polynomial retaining property of degree zero.
	\end{enumerate}

	The following two observations are straightforward.

	\begin{remark} 
		If for every vertex $x\in G$ and every integer $n\geq 0$ there is a retaining {$(\{f_n\},r)$-strategy for the initial fire $X_0=B_G(x, n)$}, then $G$  has the $(\{f_n\},r)$-retaining property.   
	\end{remark}

	\begin{remark}\label{rem:rimplies1}
		If $G$ has the $(\{f_n\},r)$-retaining property, then it has the $(\{f_n\},1)$-retaining property.  
	\end{remark}

	The following lemma is a partial converse to Remark~\ref{rem:rimplies1}.
	
	\begin{lemma}\label{lem:FireReach}
		If $G$ has the $(\{f_n\},1)$-retaining property, then it has the $(\{a_n\},r)$-retaining property where \[a_n = f_{(n-1)r+1} +\cdots + f_{nr}.\]
		Specifically, if $\{W_n\}_{n\geq 1}$ is a retaining $(\{f_n\},1)$-strategy for $X_0$ then $\{V_n\}_{n\geq1}$, where $V_n = W_{(n-1)r+1} \cup \cdots \cup W_{nr}$, is a retaining $(\{a_n\}, r)$-strategy for $X_0$. 
	\end{lemma}

	\begin{proof}
		Let $X_n$ denote the set of vertices on fire at time $n$ with respect to the retaining $(\{f_n\},1)$-strategy $\{W_n\}_{n\geq 1}$ for $X_0$. Without loss of generality, assume that $X_n\cap W_{n+1}=\emptyset$ for all $n$; see Lemma~\ref{lem:efficient}. Let $Y_0=X_0$, and let $Y_n$ denote the set of vertices on fire at time $n$ with respect to the strategy $\{V_n\}_{n\geq 1}$ for the initial fire $Y_0$ of reach $r$. It is immediate that $\abs{V_n} \leq a_n$.   
		Observe that \[ Y_n=X_{rn}\] for every $n$. 
		Hence $\bigcup_{n=0}^\infty X_n =\bigcup_{n=0}^\infty Y_n $, and therefore  $\{V_n\}_{n\geq1}$ is a retaining $(\{a_n\}, r)$-strategy for $Y_0=X_0$.
	\end{proof}

	\subsection{Finite-step retaining property} 

	An $\{f_n\}$-strategy $\{W_n\}_{n\geq 1}$ is a \emph{finite-step retaining $\{f_n\}$-strategy for $X_0$} if $\{W_n\}_{n\geq 1}$ is a retaining $\{f_n\}$-strategy for the initial fire $X_0$ of reach one, and $W_n=\emptyset$ for all sufficiently large $n$.   The graph $G$ has the  \emph{finite-step polynomial retaining property of degree $d$} if there is a  polynomial sequence $\{f_n\}$ of degree $d$ such that every finite subset $X_0$ of vertices of $G$ admits a finite-step retaining $\{f_n\}$-strategy. The finite-step polynomial retaining property of degree zero is abbreviated as the \emph{finite-step retaining property}. 

	Observe that the finite-step retaining property of degree $d$ implies the finite-step retaining property of degree $d'$ for any $d' \geq d$. However, we will see in Remark~\ref{rem:differentd} that the converse implication does not hold.

	\subsection{Containment property}\label{def:firegame2}  

	An $\{f_n\}$-strategy  $\{W_n\}_{n\geq 1}$ is  a  \emph{containment $\{f_n\}$-strategy for the initial fire $X_0$ of reach $r$} if $\bigcup_{n=0}^\infty X_n$ is finite.  The graph $G$ has the \emph{$(\{f_n\}, r)$-containment  property} if for every finite subset $X_0$ of vertices of $G$ there is a containment $\{f_n\}$-strategy for $X_0$ as an initial fire of reach $r$. The containment property on infinite graphs has been studied in~\cite{DMT17}.  Similarly as above, a graph $G$ has \emph{polynomial containment of degree $d$} if there is $K>0$ such that $G$ has the $(\{Kn^d\},1)$-containment property; $G$ has the $f$-containment property if it has the $(\{f_n\},1)$-containment property for the constant sequence $f_n=f$;  and $G$ has the \emph{constant containment property} if  $G$ has polynomial containment property of degree zero. Observe that containment strategies are in particular (finite-step) retaining strategies.

	\begin{figure}[!h]

	\centering

		\begin{tikzpicture}[scale=0.98]

		\small

		\definecolor{red}{RGB}{200,000,000}
		\definecolor{blue}{RGB}{000,000,200}

			\draw [rounded corners=15] (0,2) rectangle ++(12,8);

				\begin{scope}[shift={(0.3,9.5)}]

					\node [blue, right]   at (0,0)  {$\text{Retaining property}$};

					\node [right]   at (0,-0.5)  {$\{f_n\}_{n \geq 0} - \text{Strategy bound}$};
					
					\node [ right]   at (0,-1)  {$\growth(G) = \growth(U)$};
					
				\end{scope}

			\draw [rounded corners=15] (0,2) rectangle ++(12,4);

				\begin{scope}[shift={(0.3,5.5)}]

					\node [blue, right]   at (0,0)  {$\text{Finite-step retaining property}$};

					\node [right]   at (0,-0.5)  {$W_n = \emptyset$ \text{ for $n \gg 0$ }};

				\end{scope}

			\draw [rounded corners=15] (0,2) rectangle ++(12,2);

				\begin{scope}[shift={(0.3,3.5)}]

					\node [blue, right]   at (0,0)  {$\text{Containment property}$};

					\node [right]   at (0,-0.5)  {$|X| < \infty $};

				\end{scope}

			\draw [rounded corners=15] (5,2) rectangle ++(7,8);

				\begin{scope}[shift={(5.3,9.5)}]

					\node [red, right]   at (0,0)  {$\text{Polynomial of degree $d>0$}$};

					\node [right]   at (0,-0.5)  {$f_n= Kn^d$};

				\end{scope}

			\draw [rounded corners=15] (8.7,2) rectangle ++(3.3, 6.5);

				\begin{scope}[shift={(9,8)}]

					\node [red, right]   at (0,0)  {$\text{Constant}$};

					\node [right]   at (0,-0.5)  {$f_n= f $};

				\end{scope}

			\begin{scope}[shift={(9,6.5)}]

				\node [violet, right]   at (0,0.2)  {$A \ast_{\mathbb{Z}}B $};

				\node [violet, right]   at (0,-1)  {$A \ast_{C} B $};

				\node [violet, right]   at (0,-1.5)  {$\text{where $C$ is finite}$};

				\node [violet, right]   at (0, -3)  {$\mathbb{Z} \text{, } \mathbb{Z}^2$,};

				\node [violet, right]   at (0,-3.5)  {$\text{growth at most}$};

				\node [violet, right]   at (0, -4)  {$\text{quadratic}$};

			\end{scope}

			\begin{scope}[shift={(5.3,8.1)}]

				\node [violet, right]   at (0,0)  {$A \ast_{C} B $};
				
				\node [violet, right]   at (0,-0.5)  {$\text{where $C$ has growth}$};
				
				\node [violet, right]   at (0,-1)  {$\text{polynomial of }$};
				
				\node [violet, right]   at (0,-1.5)  {$\text{degree } d-1$};

			\end{scope}

			\begin{scope}[shift={(5.3,3.5)}]		

				\node [violet, right]   at (0,2)  {$\emptyset$};

				\node [violet, right]   at (0,0)  {$\mathbb{Z}^{d+2},$};
				
				\node [violet, right]   at (0,-0.5)  {$\text{growth polynomial}$};
				
				\node [violet, right]   at (0,-1)  {$\text{of degree } d+2$};

			\end{scope}

		\end{tikzpicture}

	\caption{The relation between retaining and containment properties and the growth of the strategy bound. The fire reach is assumed to be one. Examples of groups satisfying various properties are presented.}
	\label{fig:map}

	\end{figure}

	\section{Quasi-isometry invariance}\label{sec:qi}  

	In this section we prove the following result. 

	\begin{theorem}\label{thm:QI2}
		Let $G$ and $H$ be connected graphs with degree bounded above by a constant $\delta$.
		Let $\phi\colon G\to H$ and $\psi\colon H\to G$ be $c$-quasi-isometries, where $c$ is a positive integer such that $\dist (u, \psi\circ \phi (u))\leq c$ for every vertex $u$ of $G$.  Suppose that $G$ has the $\{f_k\}_{k\geq 1}$-retaining property. Then $H$ has the $\{b_k\}_{k\geq 1}$-retaining property where
	 	\begin{equation}\label{eq:proofQI} 
	 		b_k = \left(f_{2c(k-1)+1}+f_{2c(k-1)+2}+\cdots +f_{2ck} \right) \delta^{c^2+2c+1}. 
	 	\end{equation}

	\end{theorem}

	\begin{remark}\label{rem:QI2}
		If the sequence $\{f_k\}_{k\geq 1}$ is non-decreasing then the  sequences $\{f_k\}_{k\geq 1}$ and $\{b_k\}_{k\geq 1}$ have equivalent growth rate in the sense of Subsection~\ref{def:growthrate}, since $f_k \leq b_k \leq  2c\delta^{c^2+2c+1} f_{2ck}$ for every $k>0$.
	\end{remark}

	The following corollary is a direct consequence of Theorem~\ref{thm:QI2} and Remark~\ref{rem:QI2}.

	\begin{corollary}\label{thm:qi}
		Let $G$ and $H$ be uniformly locally finite connected graphs. 
		Suppose that $G$ is quasi-isometric to $H$.  If $G$ has polynomial retaining property of degree $d$  then $H$ has polynomial retaining property of degree $d$.
	\end{corollary}

	The proof of Theorem~\ref{thm:QI2} relies on the following statement proved in~\cite[page~18]{DMT17}. The proof is transcribed below for the convenience of the reader. 

	\begin{lemma}\cite[Lemma~4.5]{DMT17}\label{prop:DMT15}
		Let $h_0$ be a vertex of $H$ and let $g_0=\psi h_0$. Let $q$ be a positive integer and let $r=c^2+2c$.  Let $\{W_k\}_{k\geq 1}$ and $\{X_k\}_{k\geq 0}$ be sequences of subsets of $V(G)$ such that for all $k\geq 0$ we have:
		\begin{enumerate}
			\item $X_0= B_G\left(g_0, 2c(q+2)\right)$,
			\item the sets $X_k$ and $W_{k+1}$ are disjoint,
			\item the set  $X_k$ consists of the vertices which are connected to a vertex in $X_{k-1}$ by a path of length at most $2c$ containing no vertices in $W_1\cup \cdots \cup W_k$.
		\end{enumerate}
		Let $Y_0=B_H(h_0,q)$, and for $k\geq 1$ define \[Q_k = \bigcup_{ g \in W_k}   B_H (\phi g, r)   \setminus Y_{k-1} \qquad \text{and} \qquad Y_{k}=  B_H(Y_{k-1}, 1) \setminus Q_{k} .\] 
		Then  for   all $k\geq 1$ we have:
		\begin{enumerate}
			\item \label{it:lemQI1} the sets $Q_k$ and $Y_{k-1}$ are disjoint and the cardinality of $Q_k$ is at most $\delta^{r}|W_k|$,
			\item \label{it:lemQI2}  if $h \in Y_k$ then $\psi h \in X_{k-1}$.
		\end{enumerate}
	\end{lemma}

	\begin{proof}
		Observe that the first statement is immediate.  
		The second statement  is proved by  induction on $k$. First let \[ r_k = 2c(q+k+2).\]
		Observe that  $X_k$ consists of vertices $g \in G$ such that there is a path from $g_0$ to $g$ of length at most $r_k$ that does not contain vertices in $W_1\cup \cdots \cup W_k$.

		\emph{Base case:} If $h \in Y_1$  then $\dist_H (h_0, h)\leq q+1$ and hence
		\[\dist_G(g_0, \psi h) \leq c(q+1)+c \leq 2c(q+2).\]
		It follows that $\psi h$ belongs to $X_0=B_G(g_0, r_0)$.

		\emph{Induction step:} Suppose $2\leq k$. The induction hypothesis is that $h \in Y_j$ implies $\psi h \in X_{j-1}$ for all $j<k$.
		Suppose $h\in Y_k$. Then there exists a path
		\[h_0, h_1, h_2, \ldots , h_\ell=h\]
		such that $\ell \leq q+k$ and no $h_i$ is in $Q_1\cup \cdots \cup Q_k$. Consider the sequence of vertices
		\[\psi h_0, \psi h_1, \psi h_2, \ldots , \psi h_\ell.\]
		Since $\dist_G (\psi h_{i-1} , \psi h_i)\leq  c\dist_H (h_i, h_{i+1}) + c =2c$,  there is a path $\gamma_i$ of length at most $2c$ from $\psi h_{i-1}$ to $\psi h_i$. Consider the path $\gamma$ from $\psi h_0$ to $\psi h$ resulting from the concatenation $\gamma_1 \cdots \gamma_\ell$. Observe that the length of $\gamma$ is at most $2c\ell \leq 2c(q+k) \leq r_{k-1}$. To conclude that $\psi h \in X_{k-1}$, it is enough to show that no vertex of $\gamma$ is in the set $W_1\cup \cdots \cup W_{k-1}$.

		Suppose there are vertices of $\gamma$ in $W_1\cup \cdots \cup W_{k-1}$. By construction,  each vertex of $\gamma$  is at distance at most $c$ from a vertex  of the form $\psi h_i \in \gamma$. Choose a vertex $g$ of $\gamma$  and a vertex of the form $\psi h_j$  of $\gamma$ (they might be the same vertex) with the following  properties:
		\begin{enumerate}
			\item $g \in  W_1\cup \cdots \cup W_{k-1}$,
			\item the subpath of $\gamma$ between $g$ and $\psi h_j$ has length at most $c$ and it has only one vertex in  $W_1\cup \cdots \cup W_{k-1}$, namely $g$.
		\end{enumerate}

		Let  $t\leq k-1$ be the smallest integer such that $g\in W_t.$ Since
		\[\dist_G(\phi g, h_j) \leq  \dist_H ( \phi g, \phi \psi h_j) + \dist_H (\phi \psi h_j, h_j) \leq c^2+2c  = r,\]
		either $h_j \in Q_t$ or $h_j \in Y_{t-1}$. The former case is impossible by the assumption on the path from $h_0$ to $h$. Therefore $h_j \in Y_{t-1}$ and then the induction hypothesis implies that $\psi h_j \in X_{t-2}$. Since the subpath of $\gamma$ between $\psi h_j$ and $g$ has no vertices in $W_1\cup \cdots \cup W_{t-1}$ and $\psi h_j \in X_{t-2}$, it follows that $g\in X_{t-1}$. This implies that $g\notin W_t$ which is a contradiction.  
	\end{proof}

	\begin{proof}[Proof of Theorem~\ref{thm:QI2}]

		Suppose that $G$ has the $\{f_k\}_{k\geq1}$-retaining property. By Lemma~\ref{lem:FireReach},  the graph $G$ has the $(\{a_k\}_{k\geq1}, 2c)$-retaining property where
		$a_k = f_{2c(k-1)+1}+f_{2c(k-1)+2}+\cdots +f_{2ck}.$ Let 
		\[b_k = a_k \delta^{c^2+2c+1}.\] 
		We claim that $H$ has the $\{b_k\}_{k\geq1}$-retaining property as a consequence of Lemma~\ref{prop:DMT15}.  Let $h_0$ be a vertex of $H$ and consider the initial fire $Y_0=B_H(h_0, q)$ where $q$ is a positive integer.  

		Let $g_0=\psi h_0$ and consider the initial fire $X_0=B_G(g_0, 2c(q+2))$ of reach $2c$ in $G$. By assumption, there is a retaining $(\{a_k\}, 2c)$-strategy $\{W_k\}_{k\geq1}$ for $X_0$. Let $X_n$ be the set of vertices on fire at time $n$ with respect to this retaining strategy and let $X=\bigcup_{n\geq0} X_n$. 

		Now consider the sequences $\{Q_k\}_{k\geq1}$ and $\{Y_k\}_{k\geq1}$ defined in the statement of Lemma~\ref{prop:DMT15}, and observe that  $Y_k$ corresponds to the set of vertices on fire in $H$ at time $k$ with respect to the $\{b_k\}$-strategy $\{Q_k\}_{k\geq1}$ and initial fire $Y_0$ of reach one. 

		Let $U=G\setminus \bigcup_{n\geq0} X_n$ and $V= H\setminus \bigcup_{n\geq0} Y_n$. 
		 Since $\{W_k\}_{k\geq1}$ is a retaining $(\{a_k\}, 2c)$-strategy for $X_0$, it follows that \[\growth(G) = \growth(U).\] 
		Since $\psi\colon H \to G$ is a quasi-isometry, in particular, the restriction $\psi\colon \psi^{-1}(U) \to U$ is a quasi-isometry (with metrics induced from $H$ and $G$ respectively). It follows that \[\growth(H)=\growth(G)\qquad \text{and} \qquad \growth(U) = \growth(\psi^{-1}(U)).\] 
		By Lemma~\ref{prop:DMT15}\eqref{it:lemQI2}, we have the inclusion $\psi^{-1}(U) \subseteq V$, and hence 
		\[\growth(\psi^{-1}(U)) \preceq \growth(V) \preceq \growth(H).\] 
		From the above relations, it follows that 
		\[ \growth(V) =\growth(H) .\]
		Therefore $\{Q_k\}_{k\geq1}$ is a retaining $\{b_k\}$-strategy for $Y_0$ in $H$.  
	\end{proof}

	As a corollary of the proof of Theorem~\ref{thm:QI2} we obtain the following.

	\begin{corollary}\label{cor:qi-fsrp}
		Let $G$ and $H$ be connected graphs with degree bounded above by a constant $\delta$. Let $\phi\colon G\to H$ and $\psi\colon H\to G$ be $c$-quasi-isometries, where $c$ is a positive integer, and such that $\dist (u, \psi \phi u)\leq c$ for every vertex $u$ of $G$. Suppose that $G$ has the finite-step $\{f_k\}_{k\geq 1}$-retaining property. Then $H$ has the finite-step $\{b_k\}_{k\geq 1}$-retaining property where $b_k$ is given by equation~\eqref{eq:proofQI}. 
	\end{corollary}

	\begin{proof}
		Suppose that $G$ has the finite-step $\{f_k\}_{k\geq 1}$-retaining property. Then by Lemma~\ref{lem:FireReach},  
		the graph $G$ has the $(\{a_n\}, r)$-retaining property where $r=2c$ and
		\[a_n = f_{(n-1)r+1} +\cdots + f_{nr}.\]
		Lemma~\ref{lem:FireReach} further implies that for any initial fire $X_0$ of reach $r$ there is a retaining $\{a_n\}$-strategy $\{W_k\}_{k\geq 1}$ with the additional property that $W_k=\emptyset$ for all $k$ large enough. 

		Consider the initial fire $B_H(h_0, q)$ for some vertex $h_0 \in H$ and some integer $q>0$. Let $g_0= \psi h_0$ and choose an $\{a_n\}$-strategy $\{W_k\}_{k\geq 1}$ for the initial fire  $X_0=B_G(g_0, 2c(q+2))$ in $G$ such that $W_k=\emptyset$ for $k$ large enough.   Then define  $\{Q_k\}_{k\geq 1}$ as in the proof of Theorem~\ref{thm:QI2} by using Lemma~\ref{prop:DMT15}. Note that the choice of $\{W_k\}_{k\geq 1}$ implies that  $Q_k=\emptyset$ for all $k$ large enough. Then the argument proving Theorem~\ref{thm:QI2} shows that  $\{Q_k\}_{k\geq 1}$ is a finite-step retaining $\{b_k\}$-strategy for $Y_0$.
	\end{proof}

	\section{Splittings over quasi-isometrically embedded subgroups}\label{seq:splitting}
	
 	A group $G$ \emph{splits over a subgroup} $C$ if either $G=A\ast_C B$ and $C$ is a proper subgroup of $A$ and $B$, or $G$ is an HNN-extension $A\ast_C$ (with no assumptions on $C$, nor on the isomorphism $\varphi\colon C \to \varphi(C) \subset A$).

	\subsection{Coarse separation in graphs} 
	Let $\Gamma$ be a connected graph with the weak topology and the edge-path metric on its vertex set. If $K$ is a subset of vertices of $\Gamma$, a connected component of $\Gamma\setminus K$ is \emph{deep} if its set of vertices is not contained in $B_\Gamma(K, r)$  for any $r>0$. We shall say that  $K$ \emph{coarsely separates $\Gamma$} if there is $R>0$ such that $\Gamma\setminus B_\Gamma(K, R)$ has at least two deep connected components. 

	\begin{lemma}\cite[Lemma~2.2]{Papa}\label{lem:Papa}
		If a finitely generated group $G$ splits over a finitely generated subgroup $C$, then $C$ coarsely separates any Cayley graph of $G$ with respect to a finite generating set. In particular, $C$ has infinite index in $G$.
	\end{lemma}

	\begin{lemma}\label{lem:Mama}
			Let $\Gamma$ be a Cayley graph  of $G$ with respect to a finite generating set. Let $C$ be a subgroup and let $l>0$ be such that $\Gamma \setminus B_G(C, \con)$ has at least two deep  components.   
		If $U$ is a   subset of vertices of $\Gamma$ 
		that contains $\bigcup_{g\in C} gD$ where $D$ is the set of vertices of a deep  component $D$ of $\Gamma\setminus B_G(C,\con)$, then $\growth(U) = \growth(G)$.
	\end{lemma}
	\begin{proof}
		Let $L$ denote $B_\Gamma(C, \con)$. By hypothesis, $L$ separates $\Gamma$ into at least two deep connected components. Let $v_0$ be a vertex in $U$ such that $\dist (v_0, L)=1$; observe that such vertex always exists.  Since the action of $C$ on $L$ has finitely many orbits of vertices, there exists a constant $K>0$ such that any vertex of $L$ can be moved by an element of $C$ into $B_{\Gamma}(v_0,K) \cap L$. To prove the lemma, we will show that 
		\begin{equation}\label{eq:lemma2}
			\abs{B_{\Gamma}\left(e, n\right)} \leq   \abs{B_{\Gamma}\left(v_0, 2n+K+1\right) \cap U}
		\end{equation}
		for every $n\geq 0$. Therefore $\growth(G) = \growth(\Gamma) \preceq \growth(U)$. 

		Since $U$ contains the vertices of a deep component $D$ of $\Gamma\setminus L$, for every $n \geq 1$ there is a vertex $v_n \in D$ such that $\dist(L,v_n)=n+1$. Let $u_n \in L$ be the vertex realizing this distance, i.e.,\ $\dist(u_n, v_n)=n+1$. Since $\bigcup_{g\in C} gD \subset U$, by multiplying $v_n$ by an element of $C$ if necessary, we can assume that $u_n \in B_{\Gamma}(v_0,K)\cap L$ and $v_n\in U$.  Hence, \[\dist (v_0, u_n) \leq K, \quad  v_n\in U, \quad  \dist(u_n, v_n) =n+1.\]
		Notice that  
		\begin{equation*}
			v_nB_{\Gamma}(e, n) = B_{\Gamma}(v_n, n) \subseteq B_{\Gamma}(v_0, 2n+ K+1)
		\end{equation*}
		and \[v_nB_{\Gamma}(e, n) = B_{\Gamma}(v_n, n) \subseteq U,\] where the last statement follows from the assumptions that $v_n\in U$ and that $\dist(v_n, L)=n+1$. Putting these two statements together yields \[v_nB_{\Gamma}(e, n)  \subseteq  B_{\Gamma}(v_0, 2n+ K+1)\cap U \] which verifies inequality~\eqref{eq:lemma2}.
	\end{proof}

\begin{proposition}\label{prop:911}
    Let $G$ be a finitely generated group that splits over a finitely generated subgroup $C$. Then there is a finite generating set $S$ of $G$ with the following property. Let  
 $\Gamma=\Gamma(G,S)$     be the Cayley graph. If
    $\Gamma \setminus B_\Gamma(C, l)$ has at least two deep components and $X$ is a connected subgraph of $\Gamma\setminus B_\Gamma(C,l)$, then there is a deep   component $D$ of $\Gamma\setminus B_\Gamma(C,l)$ such that the   $\bigcup_{g\in C} gD$ has no vertex in $X$.
\end{proposition}
\begin{proof} 
If $G=A\ast_C B$ where $C$ is a proper subgroup of both factors, let $S$ be the union of finite   generating sets of $A$ and  $B$. If $G=A\ast_C = (A\ast \langle t\rangle)/\nclose{tat^{-1}\varphi(a)\colon a\in A}$, let $S$ be a generating set for $A$ together with the stable letter $t$.  Denote by $\dist$ the word-metric on $G$ induced by $S$.

To prove the statement,  we construct a coloring  of  the vertices of $\Gamma\setminus B_\Gamma(C,l)$ with two colors such that 
\begin{enumerate}
\item the coloring is $C$-equivariant, 
\item any 
connected subgraph of 
$\Gamma\setminus B_\Gamma(C,l)$ is monochromatic, and 
\item there is at least one deep component of each color.
\end{enumerate}
Assuming that we have such coloring, if $D$ is a deep component of color different that the color of $X$, then by $C$-equivariance, all vertices in  $\bigcup_{g\in C}gD$ have the same color, and the statement of the proposition follows. 

To define the coloring, we use the  barycentric subdivision of the (geometric realization of the) Bass-Serre tree $T$ of the splitting,  and endow it with the edge-path metric $\dist_T$. Let us recall a description of $T$, for details see~\cite{Trees}. If  $G=A\ast_C B$, let $G/A$ denote the $G$-set of left cosets of $A$, and let $G/B$ and $B/C$ denote the analogous $G$-sets. Then  $T$ is the graph with vertex set  $G/A \sqcup G/B \sqcup G/C$  and edge set   the disjoint union of $\{ \{gC, gA\}\colon g\in G \}$ and $\{ \{gC, gB\}\colon g\in G \}$. In the case that $G=A\ast_C = (A\ast \langle t\rangle)/\nclose{tat^{-1}\varphi(a)\colon a\in A}$, then $T$
has vertex set  $G/A\sqcup G/C$, and edge set the disjoint union of  $\{ \{gA, gt^{-1}Ct\} \colon g\in G \}$  and $\{ \{gA, gC\}   \colon g\in G  \}$.

\emph{Definition of the $G$-equivariant coloring.} Removing the degree two vertex $C$ of $T$, splits $T$ into two connected components, say the red one and the blue one. Let $\rho\colon G \to T$ be the $G$-equivariant map given by $g\mapsto gC$. Then each vertex of $\Gamma\setminus B_\Gamma(C,l)$ is  assigned a color, either red or blue according to its $\rho$-image. 

\emph{Connected subgraphs of $\Gamma\setminus B_\Gamma(C,l)$ are monochromatic.} The choice of $S$ implies that for any $s\in S$, $\rho(1)=\rho(s)$ if $s\in C$ and $\dist_T(\rho(1),\rho(s))=2$ if $s\not\in C$. Hence an edge of $\Gamma$ between $g$ and $gs$ with $s\not\in C$ induces a path of length two in $T$ from $gC$ to $gsC$ with middle vertex distinct than $C$. Therefore, any path in $\Gamma$ from $x$ to $y$ that does not pass through a vertex in $C$ induces a path in $T$ from $\rho(x)$ to $\rho(y)$ that does not pass through the vertex $C$ and of at most twice the length. In particular any path (and hence any connected subgraph) in $\Gamma\setminus B_\Gamma(C,l)$ is monochromatic, and
\begin{equation}\label{eq:final911}
  2\dist_G(x,y)\geq \dist_T(\rho(x),\rho(y)) \quad\text{ and }\quad \dist_T(\rho(x),\rho(xs))=2    
\end{equation}
for any $x,y\in G$ and $s\in S\setminus C$.

\emph{The subgraph $\Gamma\setminus B_\Gamma(C,l)$ has blue and red deep components.} 
Since $G$ splits over $C$,  both components of $T\setminus\{C\}$ are infinite trees. Note that the pre-image by $\rho\colon G\setminus B_\Gamma(C,l) \to T$ of an infinite ray in one of the components of $T\setminus\{C\}$ spans a  connected subgraph of $G\setminus B_\Gamma(C,l)$, hence it is monochromatic, and by~\eqref{eq:final911} determines a deep component. Taking an infinite ray in each component of $T\setminus\{C\}$ shows that there are deep components of each color.
\end{proof}

\subsection{Group splittings imply retaining}

	\begin{proposition}\label{cor:split-finite}
		Let $G$ be a finitely generated group that splits over a finite subgroup. 
		If $\Gamma$ is the Cayley graph of $G$ with respect to a finite generating set, then $\Gamma$ has the finite-step retaining property. 
	\end{proposition}

	\begin{proof}
	Suppose $G$ splits over a finite subgroup $C$. In view of Corollary~\ref{cor:qi-fsrp} is enough to consider $\Gamma$ to be the Cayley graph  of $G$ with respect to a finite generating set $S$ provided by Proposition~\ref{prop:911}. Denote by $\dist$ its edge-path metric. 
		Since $G$ splits over $C$, by Lemma~\ref{lem:Papa}, there is a constant $\con>0$ such that the $\con$-neighborhood $L$ of $C$ in $\Gamma$, \[ L=\{g\in G \colon \dist (g, C) \leq \con \}, \] separates $\Gamma$ into at least two deep components.
		Let $f$ be the cardinality of $L$. Let $X_0$ be a finite subset of $G$. Since $C$ has infinite index in $G$, there is $g\in G$ such that $\dist(gL, X_0) \geq \diam X_0$. The inequality $\dist(gL, X_0) \geq \diam X_0 $ implies that there is a deep component of $\Gamma \setminus gL$ that does not intersect $X_0$. Consider the strategy $\{W_n\}_{n\geq 1}$ where $W_1=gL$ and $W_n=\emptyset$ for $n>1$.  
		Let $X_n$ denote the set of vertices on fire at time $n$ with respect to this strategy and the initial fire $X_0$ of reach one.
		Observe that $X_n\cap W_{1}=\emptyset$ for all $n\geq 0$. Hence $X=\bigcup_{n\geq 0} X_n$ spans a connected subgraph of $G\setminus L$. By Proposition~\ref{prop:911}, there is a deep  component $D$ of $\Gamma\setminus L$ such that $\bigcup_{g\in G}gD$ has no vertex in $X$.
		Let $U=G\setminus X$. By Lemma~\ref{lem:Mama}, we have that $\growth(G) = \growth(U)$ and hence $\{W_n\}_{n\geq 1}$ is a finite-step retaining $f$-strategy for $X_0$.
	\end{proof}

	\begin{theorem}\label{thm:splitpoly}
		Let $G$ be a finitely generated group that splits over a finitely generated subgroup $C$. Suppose that $C$ has polynomial growth of degree $d>0$, and is quasi-isometrically embedded into $G$. If $\Gamma$ is the Cayley graph of $G$ with respect to a finite generating, then $\Gamma$ has polynomial retaining property of degree $d-1$.
	\end{theorem}

	\begin{proof}
	By Corollary~\ref{cor:qi-fsrp}, it is enough to prove the statement for $\Gamma$ the Cayley graph with respect a finite generating set $S$ provided by Proposition~\ref{prop:911}. 	Let $\dist$ denote the word-metric on $G$ with respect to $S$. Since $G$ splits over $C$, there is a constant $\con>0$ such that the $\con$-neighborhood $L$ of $C$ in $\Gamma$, 
		\[ L=\{g\in G \colon \dist  (g, C) \leq \con \}, \]
		separates $\Gamma$ into at least two deep components, see Lemma~\ref{lem:Papa}.

		\begin{step}\label{step:growthL}
			There is a constant $K_1>0$ such that for any $g\in G$, for any $y_0 \in gL$, and for any $n>0$ we have \[\beta_{gL, y_0} (n) \leq K_1 n^d\] where $\beta_{gL, y_0}$ is the growth function of the metric space $(gL, \dist)$. 
		\end{step}
		\begin{proof}[Proof of Step~\ref{step:growthL}]
			Let $\dist_C$ denote a word-metric on $C$ with respect to a finite generating set of $C$. The assumption that $C$ is quasi-isometrically embedded in $G$ means that  the  spaces $(C, \dist )$ and $(C, \dist_C)$ are quasi-isometric. It follows that $(C, \dist_C)$,  $(C, \dist )$, $(L, \dist )$ are all quasi-isometric. Since they all are discrete uniformly proper metric spaces, by Remark~\ref{rem:growth}\eqref{it:growth3} they all have polynomial growth of degree $d$.  Since $C$ acts by isometries and cocompactly on $(L, \dist)$,  there exists a constant $K_1>0$ such that for any choice of basepoint on $L$, the corresponding growth function of $(L, \dist)$ is bounded from above by $K_1n^d$. Since the spaces $(L, \dist )$ and $(gL, \dist )$ are isometric, the statement follows.
		\end{proof}

		\begin{step}\label{lem:diam-est}
			Let $K=2^dK_1$.  Let $X_0$ be a finite subset of $G$, $g$ an element of $G$, and $n$ a positive integer. Define 
			\[M_{n, g, X_0} = \{x\in gL \colon \dist (x, X_0) \leq n\}.\]
			Then 
			\[ \abs{M_{n, g, X_0}} \leq  K(n  +\diam X_0)^d. \]
		\end{step}
		Note that $M_{n, g, X_0}$ is the set of vertices of $gL$ that would be on fire by the time $n$ if the initial fire was $X_0$ and no vertices were protected.
		\begin{proof}[Proof of Step~\ref{lem:diam-est}]
			By Step~\eqref{step:growthL},  the growth function of $(gL, \dist)$ with respect to any basepoint is bounded by $K_1n^d$. Let $y_0$ be an element of $gL$ such that $\dist (y_0, X_0) = \dist  (gL, X_0)$. 
			The triangle inequality implies that 
			\[ \diam M_{n, g, X_0} \leq  2n +\diam X_0,\]
			and hence  $M_{n, g, X_0}$ is contained in the ball $B_{gL}(y_0, 2n+\diam X_0)$. To conclude, observe that
			\[ \abs{M_{n, g, X_0}} \leq  \abs{B_{gL}(y_0, 2n+\diam X_0 )} \leq K_1(2n+\diam X_0 )^d. \qedhere\]
		\end{proof}

		\begin{step}\label{step:3}
			Let $F=K+1$. Let $X_0$ be a finite subset of $G$. Then there is $g\in G$ such that 
			\begin{equation}\label{eq:choiceg3}
				\diam X_0 <\dist(gL, X_0),
			\end{equation}
			and for every $n>0$
			\begin{equation}\label{eq:GeneralM2}
				\abs{M_{n, g, X_0}}  <  \sum_{k=1}^n dF k^{d-1}.
			\end{equation}
		\end{step}
		\begin{proof}[Proof of Step~\ref{step:3}]
			By enlarging $X_0$ if necessary, we can assume that it contains the identity element of $G$. 
			Since $X_0$ is finite and the index of $C$ in $G$ is infinite, we can choose $g\in G$ such that $\dist(gL, X_0)$ is large enough to guarantee that both inequality~\eqref{eq:choiceg3} and the following inequality are satisfied.
			\begin{equation*}
				K(\dist(gL, X_0)  +\diam X_0)^d <  F (\dist(gL, X_0) )^d.
			\end{equation*}
			This inequality together with the statement of Step~\ref{lem:diam-est} implies that
			\begin{equation*}
				\abs{M_{\dist(gL, X_0), g, X_0}} <  F (\dist(gL, X_0))^d.
			\end{equation*}
			Since $M_{n, g, X_0}$ is empty for $n<\dist (gL, X_0)$ and $F>K$, it follows that $ \abs{M_{n, g, X_0}}  <  F n^d $ for every $n\in \N$. A calculus exercise shows that $n^d \leq d \sum_{k=1}^n k^{d-1}$, and thus inequality~\eqref{eq:GeneralM2} is satisfied. 
		\end{proof}

		Inequality~\eqref{eq:GeneralM2} allows us to define a retaining $\{dFn^{d-1}\}$-strategy for any finite subset $X_0$ of $G$; this is proved in the next step concluding the proof of the theorem.  To simplify the notation define \[p_n= \sum_{k=1}^n dF k^{d-1},\] and observe that $p_n$ is the maximal number of vertices that can be protected by the time $n$ using a $\{dFn^{d-1}\}$-strategy. 

		\begin{step}\label{step:4}
			Let $X_0$ be a finite subset of $G$. Let $g\in G$ be an element satisfying inequalities~\eqref{eq:choiceg3} and~\eqref{eq:GeneralM2}.  Let $w_1,w_2, w_3, \ldots$ be an enumeration of the countable set $gL$ such that the sequence $\{\dist(w_i, X_0) \}_{i\geq 1}$ is non-decreasing, and for each integer $n\geq 1$ let
			\[ W_n = \Bigg \{w_i \colon p_{n-1}<i\leq p_n \text{ and } w_i \not \in \bigcup_{1\leq i<n} W_i \Bigg \}. \]
			Then $\{W_n\}_{n\geq 1}$ is a retaining $\{dFn^{d-1}\}$-strategy for $X_0$.
		\end{step}
		\begin{proof}[Proof of Step~\ref{step:4}]
			Observe that
			\[ \abs{W_n} \leq p_n-p_{n-1} = dFn^{d-1} \quad \text{ for every  $n \geq 0$}, \]
			and hence $\{W_n\}_{n\geq 1}$ is a  $\{dFn^{d-1}\}_{n\geq1}$-strategy. Let $X_n$ denote the set of vertices on fire at time $n$ with respect to this strategy and the initial fire $X_0$ of reach one. We claim that 
			\begin{equation} \label{eq:disjointWX} 
				X_n \cap W_{n+1} = \emptyset, \quad \text{ for all $n\geq 0$}.
			\end{equation}
		 	Indeed, observe that $X_0$ and $W_1$ are disjoint as a consequence of inequality~\eqref{eq:choiceg3}. Suppose, by induction, that $X_{n-1}$ has been defined, and $X_{n-1} \subseteq  B_\Gamma(X_0, n-1)$, and $X_{n-1}$ and $W_{n}$ are disjoint. Recall that $X_n$ consists of vertices $v$ such that $\dist (v, X_{n-1}) \leq 1$ and $v\not \in W_1\cup \cdots \cup W_n$. Thus $X_n \subseteq  B_\Gamma(X_0, n) $.  Since  $W_{n+1} \subseteq gL$, we have	\[X_n\cap W_{n+1} \subseteq X_n \cap gL   \subseteq  B_\Gamma(X_0, n) \cap gL =  M_{n, g, X_0} \subseteq \bigcup_{i=1}^n W_i,\]
		 	where the last inclusion is a consequence of~\eqref{eq:GeneralM2} and the definition of the $W_i$'s.
			By definition, $W_{n+1} \cap \bigcup_{i=1}^n W_i = \emptyset$, and therefore  $X_n \cap W_{n+1} = \emptyset.$ This concludes the verification of equation~\eqref{eq:disjointWX}.

            By definition $X_n \subset X_{n+1}$ for all $n\geq 0$ and $X_n\cap W_m=\emptyset$ if $n\geq m$. Let $X=\bigcup_{n\geq 1}X_n$ and observe that~\eqref{eq:disjointWX} implies that
            \[ X\cap L = \bigcup_{n\geq 0} X_n \cap \bigcup_{n\geq 1} W_n = \emptyset.\]
            Since $X$ spans  a connected subgraph of $\Gamma\setminus L$,  the choice of the finite generating set $S$ given by Proposition~\ref{prop:911} implies that there is a deep component $D$ of $\Gamma\setminus L$ such that $\bigcup_{g\in G} gD$ has no vertex in $X$. 	Let $U= G \setminus X$ and note that Lemma~\ref{lem:Mama} implies that  $\growth(G) = \growth(U)$, and hence $\{W_n\}_{n\geq 1}$ is a retaining $\{dFn^{d-1}\}$-strategy for $X_0$.
		\end{proof} 

		Step~\ref{step:4} concludes the proof of the theorem.
	\end{proof}

	\subsection{Ends of groups and the finite-step retaining property}

	The following proposition uses the notion of an \emph{end} of a topological space. For a definition of an end we refer the reader to~\cite{BH99}, and we follow the convention that a graph carries the weak topology.

	\begin{proposition}\label{prop:FiniteStepEnds}
		Let $G$ be a locally finite connected graph. If $G$ has the finite-step retaining property of degree $d$, then either $G$ has the containment property of degree $d$, or $G$ has at least two ends. 
	\end{proposition}
	\begin{proof}
		Let $K>0$ be a constant such for every finite subset of vertices of $G$ there is a finite-step retaining $\{Kn^d\}$-strategy. Suppose that $G$ does not have the containment property of degree $d$. In particular, this implies that $G$ has infinitely many vertices. Then there is an initial fire $X_0$ of reach one for which there is no finite-step retaining $\{Kn^d\}$-strategy that contains it. Let $\{W_n\}_{n\geq 1}$ be a  finite-step retaining $\{Kn^d\}$-strategy for the initial fire $X_0$. Let $X_n$ be the set of vertices on fire at time $n$ with respect to this strategy. Let 
		\begin{equation*}
			X=\bigcup_{n\geq0} X_{n} \quad \text{ and } \quad W=\bigcup_{n\geq1} W_n,
		\end{equation*} and let $U$ be the complement of $X$ in the set of vertices of $G$. 
		Since $\{W_n\}_{n\geq 1}$ is a finite-step retaining strategy for $X_0$, the set $W$ is finite. We claim that $G\setminus W$ contains at least two unbounded connected components. 

		Since $X_0$ is not contained by the strategy $\{W_n\}_{n\geq 1}$, the set $X$ is infinite.   By definition of $X_n$, see Subsection~\ref{def:Xn}, every vertex of $X$ is connected to a vertex of $X_0$ by a path in $G$ that contains only vertices in $X$. Since $X_0$ is finite and $X$ is infinite, the subgraph $A$ of $G$ spanned by $X$ contains an infinite connected subgraph that we denote by $A'$. 

		Since $\{W_n\}_{n\geq 1}$ is a retaining strategy for $X_0$, it follows that $\growth(G)=\growth(U)$. Since $G$ is connected and has infinitely many vertices, $U$ is an infinite subset of vertices.
		Let $U'=U\setminus W$.  
		By definition of $X_n$ every path in $G$ between a vertex in $X$ and a vertex in $U'$ contains a vertex in $W$.
		Let $B$ be the subgraph of $G$ spanned by $U'$.
		Consider the map from the collection of connected components of $B$ to the collection of non-empty subsets of $W$, that assigns to a connected component the subset of elements of $W$ that appear in minimal length paths from a vertex in the component  to a vertex in $X$. Since the graph $G$ is locally finite, this map is finite to one. Therefore the number of connected components of $B$ is finite. Since $U'$ is infinite, $B$ contains an infinite connected subgraph that we denote by $B'$. 

		Because paths between $X$ and $U'$ have to pass through $W$, we have that $A'$ and $B'$ are contained in different connected components of $G\setminus W$. Since $G$ is locally finite, the infinite connected subgraphs $A'$ and $B'$ are unbounded.  Therefore $G\setminus W$ contains at least two unbounded connected components which implies that $G$ has at least two ends. 
	\end{proof}

	\begin{corollary}\label{cor:Tomasz}
		Let $G$ be a finitely generated group. If $G$ has the finite-step retaining property, then either $G$ has the constant containment property or $G$ has infinitely many ends. 
	\end{corollary}
	\begin{proof}
		Suppose that $G$ does not have the constant containment property.  Finitely generated groups with two ends are virtually cyclic and hence they have linear growth~\cite[Part~I, Theorem~8.32(3) and Example~8.36]{BH99}. Since finitely generated groups with growth at most quadratic have the constant containment property~\cite[Theorem~1]{DMT17}, it follows that $G$ does not have two ends. On the other hand, a finitely generated group has either $0$, $1$, $2$, or infinitely many ends~\cite[Part~I, Theorem~8.32(1)]{BH99}. Therefore Proposition~\ref{prop:FiniteStepEnds} implies that $G$ has infinitely many ends.
	\end{proof}
	  
	\begin{remark}\label{rem:differentd}
		The finite-step retaining property of degree $d+1$ is not equivalent to the finite-step retaining property of degree $d$ for $d \geq 0$. Indeed, consider the group $ G =\mathbb{Z}^{d+3}$. This group has containment property of degree $d+1$, see~\cite[Theorem~3]{DMT17}. In particular $G$ has the finite-step retaining property of degree $d+1$. However, by~\cite[Corollary~6]{DMT17}, the group $G$ does not have the containment property of degree $d$. Since $G$ is one-ended, by Proposition~\ref{prop:FiniteStepEnds}, $G$  does not have the finite-step retaining property of degree $d$.
	\end{remark}

	\bibliographystyle{plain}
	\bibliography{xbib}

\end{document}